\documentclass[reqno,11pt]{article}%
\usepackage{amsmath}
\usepackage{graphicx}
\usepackage{amsfonts}
\usepackage{amssymb}%
 \makeatletter
\newtheorem{theorem}{Theorem}

\newtheorem{corollary}[theorem]{Corollary}

\newtheorem{lemma}[theorem]{Lemma}

\newenvironment{proof}[1][Proof]{\noindent{\textbf {#1}  }}  {\hfill$\Box$\bigskip}

\begin{document}

\title{Extremal problems for the $p$-spectral radius of graphs\thanks{\textbf{AMS
MSC:} 05C50; 05C35\textit{.}} \thanks{\textbf{Keywords:} \textit{ }%
$p$-\textit{spectral radius; clique number; extremal problems; Tur\'{a}n
problems; saturation problems. }}}
\author{L. Kang\thanks{Department of Mathematics, Shanghai University, Shanghai, PR
China. email: \textit{lykang@shu.edu.cn }} \ and V.
Nikiforov\thanks{Department of Mathematical Sciences, University of Memphis,
Memphis TN 38152, USA; email: \textit{vnikifrv@memphis.edu}}}
\maketitle

\begin{abstract}
The $p$-spectral radius of a graph $G\ $of order $n$ is defined for any real
number $p\geq1$ as
\[
\lambda^{\left(  p\right)  }\left(  G\right)  =\max\left\{  2\sum_{\{i,j\}\in
E\left(  G\right)  \ }x_{i}x_{j}:x_{1},\ldots,x_{n}\in\mathbb{R}\text{
and\ }\left\vert x_{1}\right\vert ^{p}+\cdots+\left\vert x_{n}\right\vert
^{p}=1\right\}  .
\]
The most remarkable feature of $\lambda^{\left(  p\right)  }$ is that it
seamlessly joins several other graph parameters, e.g., $\lambda^{\left(
1\right)  }$ is the Lagrangian, $\lambda^{\left(  2\right)  }$ is the spectral
radius and $\lambda^{\left(  \infty\right)  }/2$ is the number of edges. This
paper presents solutions to some extremal problems about $\lambda^{\left(
p\right)  }$, which are common generalizations of corresponding edge and
spectral extremal problems.

Let $T_{r}\left(  n\right)  $ be the $r$-partite Tur\'{a}n graph of order $n.$
Two of the main results in the paper are:

(I) Let $r\geq2$ and $p>1.$ If $G$ is a $K_{r+1}$-free graph of order $n,$
then
\[
\lambda^{\left(  p\right)  }\left(  G\right)  <\lambda^{\left(  p\right)
}\left(  T_{r}\left(  n\right)  \right)  ,
\]
unless $G=T_{r}\left(  n\right)  .$

(II) Let $r\geq2$ and $p>1.$ If $G\ $is a graph of order $n,$ with
\[
\lambda^{\left(  p\right)  }\left(  G\right)  >\lambda^{\left(  p\right)
}\left(  T_{r}\left(  n\right)  \right)  ,
\]
then $G$ has an edge contained in at least $cn^{r-1}$ cliques of order $r+1,$
where $c$ is a positive number depending only on $p$ and $r.$

\end{abstract}

\section{\label{Def}Introduction}

In this paper we study extremal problems for the $p$-spectral radius
$\lambda^{\left(  p\right)  }$ of graphs, so first let us recall the
definition of $\lambda^{\left(  p\right)  }$. Suppose that $G$ is a graph of
order $n$. The quadratic form of $G$ is defined for any vector $\left[
x_{i}\right]  \in\mathbb{R}^{n}$ as
\[
P_{G}\left(  \left[  x_{i}\right]  \right)  :=2\sum_{\left\{  i,j\right\}  \in
E\left(  G\right)  }x_{i}x_{j}.
\]
Now, for any real number $p\geq1,$ the $p$\emph{-spectral radius} of $G$ is
defined as
\[
\lambda^{\left(  p\right)  }\left(  G\right)  =\max\left\{  2\sum_{\{i,j\}\in
E\left(  G\right)  \ }x_{i}x_{j}:x_{1},\ldots,x_{n}\in\mathbb{R}\text{
and\ }\left\vert x_{1}\right\vert ^{p}+\cdots+\left\vert x_{n}\right\vert
^{p}=1\right\}  .
\]
Note that $\lambda^{\left(  p\right)  }$ is a multifaceted parameter, as
$\lambda^{\left(  1\right)  }\left(  G\right)  $ is the Lagrangian of $G,$
$\lambda^{\left(  2\right)  }\left(  G\right)  $ is its spectral radius, and
$\lim_{p\rightarrow\infty}\lambda^{\left(  p\right)  }\left(  G\right)
=2e\left(  G\right)  $. The $p$-spectral radius has been introduced for
uniform hypergraphs by Keevash, Lenz, and Mubayi in \cite{KLM13}, and
subsequently studied in \cite{KNY14}, \cite{Nik14}, \cite{NikA},\ and
\cite{NikB}.\medskip

The problems studied in this paper originate from the following general
one:\medskip

\emph{What is the maximum }$\lambda^{\left(  p\right)  }\left(  G\right)
$\emph{ of a graph }$G$\emph{ of order }$n,$\emph{ not containing a given
subgraph }$H$\emph{?}\medskip

Similar questions for the maximum number of edges $e\left(  G\right)  $ and a
fixed subgraph $H$ are called Tur\'{a}n problems and are central in classical
extremal graph theory, as known, e.g., from \cite{Bol78}, Ch. 6. In fact, we
shall build a parallel extremal theory for $\lambda^{\left(  p\right)  },$
which extends the classical theory, given that $\lim_{p\rightarrow\infty
}\lambda^{\left(  p\right)  }\left(  G\right)  =2e\left(  G\right)  ;$ thus,
the classical extremal theory is a limiting case of the extremal theory for
$\lambda^{\left(  p\right)  }$. More important, our main focus will be on
forbidden subgraphs $H$ whose order grows with $n,$ as this approach gives
more insight and leads to definite results like Theorem \ref{sE}
below.\medskip

To begin with, recall that the Tur\'{a}n graph $T_{r}\left(  n\right)  $ is
the complete $r$-partite graph of order $n,$ with parts of size $\left\lfloor
n/r\right\rfloor $ or $\left\lceil n/r\right\rceil .$ The prominence of
$T_{r}\left(  n\right)  $ in extremal graph theory has been established by the
ground-breaking result of Tur\'{a}n \cite{Tur41}:\medskip

\textbf{Theorem A }\emph{If }$G$\emph{ is a }$K_{r+1}$\emph{-free graph of
order }$n,$\emph{ then }$e\left(  G\right)  <e\left(  T_{r}\left(  n\right)
\right)  ,$\emph{ unless }$G=T_{r}\left(  n\right)  .$\medskip

A very similar result has been proved for the spectral radius $\lambda
^{\left(  2\right)  }$ in \cite{Nik07}:\medskip

\textbf{Theorem B }\emph{If }$G$\emph{ is a }$K_{r+1}$\emph{-free graph of
order }$n,$\emph{ then }$\lambda^{\left(  2\right)  }\left(  G\right)
<\lambda^{\left(  2\right)  }\left(  T_{r}\left(  n\right)  \right)  ,$\emph{
unless }$G=T_{r}\left(  n\right)  .$\medskip

Our starting point is a common generalization of Theorems A and B, stated as follows.

\begin{theorem}
\label{Tur}Let $r\geq2$ and $p>1.$ If $G$ is\emph{ }$K_{r+1}$-free graph of
order $n$, then $\lambda^{\left(  p\right)  }\left(  G\right)  <\lambda
^{\left(  p\right)  }\left(  T_{r}\left(  n\right)  \right)  ,$ unless
$G=T_{r}\left(  n\right)  .$
\end{theorem}

\medskip

Like Tur\'{a}n's theorem in extremal graph theory, Theorem \ref{Tur} motivates
a lot of related results, some of which we shall study in this and a
forthcoming paper. In particular, our results answer important instances of
the following broad question:\medskip

\emph{Which subgraphs are necessary present in a graph }$G$\emph{ of
sufficiently large order }$n$\emph{ if}
\[
\lambda^{\left(  p\right)  }\left(  G\right)  >\lambda^{\left(  p\right)
}\left(  T_{r}\left(  n\right)  \right)  ?
\]
As we shall see, here the range of the difference $f\left(  n\right)
=\lambda^{\left(  p\right)  }\left(  G\right)  -\lambda^{\left(  p\right)
}\left(  T_{r}\left(  n\right)  \right)  $ determines different problems: when
$f\left(  n\right)  =o\left(  n^{1-2/p}\right)  $ we have what are called
\emph{saturation problems}, and when $f\left(  n\right)  =O\left(
n^{1-2p}\right)  $, we have \emph{Erd\H{o}s-Stone type problems}.\smallskip

We also shall study \emph{stability problems, }which concern near-maximal
graphs without forbidden subgraphs. More precisely a stability problem can be
stated as:\smallskip

\emph{Suppose that }$H$\emph{ is a graph which is necessary present in any
graph }$G$\emph{ of sufficiently large order }$n,$ \emph{with }$\lambda
^{\left(  p\right)  }\left(  G\right)  >\lambda^{\left(  p\right)  }\left(
T_{r}\left(  n\right)  \right)  $\emph{. What is the structure of a graph }%
$G$\emph{ of order }$n$ \emph{if}
\[
\lambda^{\left(  p\right)  }\left(  G\right)  >\lambda^{\left(  p\right)
}\left(  T_{r}\left(  n\right)  \right)  -o\left(  n^{2-2/p}\right)  ,
\]
\emph{but }$G$ \emph{contains no }$H?$\medskip

Many extremal problems along the above lines have been successfully solved for
$\lambda^{\left(  2\right)  },$ the classical spectral radius; see
\cite{Nik11} for a survey and references. However, $\lambda^{\left(  2\right)
}$ belongs to the realm of Linear Algebra and its study builds on proven solid
ground. By contrast, linear-algebraic methods are irrelevant for the study of
$\lambda^{\left(  p\right)  }$ in general, and in fact no efficient general
methods are known for it. Thus the study of $\lambda^{\left(  p\right)  }$ for
$p\neq2$ is far more complicated than of $\lambda^{\left(  2\right)  }.$ One
of the aims on the present study is to find out if specific applications of
the spectral radius $\lambda^{\left(  2\right)  }$ can be extended to
$\lambda^{\left(  p\right)  }$ in general. So far, most attempts have been
successful, but there are many basic unanswered questions, see \cite{NikB} for
some examples.

It should be noted that extremal problems for $\lambda^{\left(  p\right)  }$
of hypergraphs have been studied in \cite{KNY14}, \cite{KLM13}, \cite{NikA},
and \cite{NikB}, but $2$-graphs are better understood, so it is worthwhile to
delve into deeper extremal theory. Another line has been investigated in
\cite{Nik14}, where the emphasis is on hereditary properties. \medskip

\section{Tur\'{a}n type theorems for $\lambda^{\left(  p\right)  }\left(
G\right)  $}

It is not hard to see that if $n\geq r>q,$ then $\lambda^{\left(  p\right)
}\left(  T_{r}\left(  n\right)  \right)  >\lambda^{\left(  p\right)  }\left(
T_{q}\left(  n\right)  \right)  $ for every $p\geq1.$ This observation entails
the following reformulation of Theorem \ref{Tur}.

\begin{theorem}
Let $r\geq2$ and $p>1.$ If $G$ is a graph of order $n$, with clique number
$\omega,$ then $\lambda^{\left(  p\right)  }\left(  G\right)  <\lambda
^{\left(  p\right)  }\left(  T_{\omega}\left(  n\right)  \right)  ,$ unless
$G=T_{\omega}\left(  n\right)  .$
\end{theorem}

As already mentioned, $\lim_{p\rightarrow\infty}\lambda^{\left(  p\right)
}\left(  G\right)  =2e\left(  G\right)  ,$ so Tur\'{a}n's Theorem A can be
recovered in full detail from Theorem \ref{Tur}.

Let us note that particular relations between the clique number $\omega$ of a
graph $G$ and $\lambda^{\left(  p\right)  }\left(  G\right)  $ have been long
known. For example, the result of Motzkin and Straus \cite{MoSt65} (see
Theorem E below) establishes the fundamental fact that $\lambda^{\left(
1\right)  }\left(  G\right)  =1-1/\omega;$ later it has been used by Wilf
\cite{Wil86} to derive
\[
\lambda^{\left(  2\right)  }\left(  G\right)  \leq\left(  1-1/\omega\right)
n;
\]
and in \cite{Nik02} it was used for the stronger inequality
\[
\lambda^{\left(  2\right)  }\left(  G\right)  \leq\sqrt{2\left(
1-1/\omega\right)  e\left(  G\right)  }.
\]

Note that the last two results are explicit, while being almost tight. It
turns out that the approach of Motzkin and Straus helps to deduce similar
explicit results for $\lambda^{\left(  p\right)  }\left(  G\right)  $ and any
$p\geq1$ as well.

\begin{theorem}
\label{Turc}Let $r\geq2$ and $p\geq1.$ If $G$ is a $K_{r+1}$-free graph of
order $n$, then
\begin{equation}
\lambda^{\left(  p\right)  }\left(  G\right)  \leq\left(  1-\frac{1}%
{r}\right)  ^{1/p}\left(  2e\left(  G\right)  \right)  ^{1-1/p}, \label{in0}%
\end{equation}
and
\begin{equation}
\lambda^{\left(  p\right)  }\left(  G\right)  \leq\left(  1-\frac{1}%
{r}\right)  n^{2-2/p}. \label{in1}%
\end{equation}
If $p>1$, equality holds in (\ref{in1}) if and only if $r|n$ and
$G=T_{r}\left(  n\right)  .$
\end{theorem}

In particular, Theorem \ref{Turc} implies that if $G$ is a graph of order $n$,
with clique number $\omega,$ then
\[
\lambda^{\left(  p\right)  }\left(  G\right)  \leq\left(  1-1/\omega\right)
^{1/p}\left(  2e\left(  G\right)  \right)  ^{1-1/p}%
\]
and
\[
\lambda^{\left(  p\right)  }\left(  G\right)  \leq\left(  1-1/\omega\right)
n^{2-2/p}.
\]

A natural question is how good bounds (\ref{in0}) and (\ref{in1}) are compared
to the bound in Theorem \ref{Tur}, which is attained for every $n$. It turns
out that bounds (\ref{in0}) and (\ref{in1}) are never too far from the best
possible one, as seen in the following several estimates.

\begin{theorem}
\label{Turc1}Let $T_{r}\left(  n\right)  $ be the $r$-partite Tur\'{a}n graph
of order $n.$ Then
\begin{equation}
\lambda^{\left(  1\right)  }\left(  T_{r}\left(  n\right)  \right)
=1-1/r,\label{l1}%
\end{equation}
and for every $p>1,$%
\begin{align}
2e\left(  T_{r}\left(  n\right)  \right)   &  \leq\lambda^{\left(  p\right)
}\left(  T_{r}\left(  n\right)  \right)  n^{2/p}\leq2e\left(  T_{r}\left(
n\right)  \right)  \left(  1+\frac{r}{pn^{2}}\right)  ,\label{le}\\
\left(  1-\frac{1}{r}\right)  n^{2}-\frac{r}{4} &  \leq\lambda^{\left(
p\right)  }\left(  T_{r}\left(  n\right)  \right)  n^{2/p}\leq\left(
1-\frac{1}{r}\right)  n^{2}.\label{lv}%
\end{align}

\end{theorem}

\section{Saturation problems}

Theorem \ref{Tur} implies that if $G$ is a graph of order $n,$ with
$\lambda^{\left(  p\right)  }\left(  G\right)  >\lambda^{\left(  p\right)
}\left(  T_{r}\left(  n\right)  \right)  ,$ then $G$ contains a $K_{r+1}.$ We
shall show that, in fact, much larger supergraphs of $K_{r+1}$ can be found in
$G.$ Such problems are usually called saturation problems.

\subsection{Joints}

In \cite{Erd69} Erd\H{o}s proved that if $r\geq2,$ and $G$ is a graph of
sufficiently large order $n,$ with $e\left(  G\right)  >e\left(  T_{r}\left(
n\right)  \right)  ,$ then $G$ has an edge that is contained in at least
$n^{r-1}/\left(  10r\right)  ^{6r}$ cliques of order $r+1.$ This fact is
fundamental, so to study its consequences, the following definition was given
in \cite{BoNi04}:\medskip

\textit{An }$r$\emph{-joint}\textit{ of size }$t$\textit{ is a collection of
}$t$\textit{ distinct }$r$\textit{-cliques sharing an edge.}\medskip

A $3$-joint is also called a \emph{book. }Books have been studied extensively
in extremal and Ramsey graph theory.\emph{ }Note that books are determined by
their size alone, while for $r>3$ there are many non-isomorphic $r$-joints of
the same size.

We write $\mathrm{js}_{r}\left(  G\right)  $ for the maximum size of an
$r$-joint in a graph $G.$ The following theorem enhances Theorem \ref{Tur},
insofar that from the same premises it implies the existence of subgraphs
whose order grows with $n.$ For this reason we shall use it as a starting
point for several other extensions.

\begin{theorem}
\label{js}Let $r\geq2$ and $p>1$. If $G$ is a graph of order $n,$ with
\[
\lambda^{\left(  p\right)  }\left(  G\right)  \geq\lambda^{\left(  p\right)
}\left(  T_{r}\left(  n\right)  \right)  ,
\]
then
\[
\mathrm{js}_{r+1}\left(  G\right)  >\frac{n^{r-1}}{r^{r^{6}p/\left(
p-1\right)  }},
\]
unless $G=T_{r}\left(  n\right)  .$
\end{theorem}

Let us note that the order of $n^{r-1}$ is obviously best possible, but the
coefficient $r^{-r^{6}p/\left(  p-1\right)  }$ is far from being optimal.
Nevertheless, this small coefficient makes the statement valid for all $n,$
and for larger $n$ it can be somewhat increased.

\subsection{Color critical subgraphs}

Call a graph $k$\emph{-color critical}, if it is $k$-colorable, but it can be
made $\left(  k-1\right)  $-colorable by removing a particular edge. For
example, books are $3$-color critical graphs.

Simonovits \cite{Sim68} has proved that if $F$ is an $\left(  r+1\right)
$-color critical graph, then $F\subset G$ for every graph $G$ of sufficiently
large order $n,$ with $e\left(  G\right)  >e\left(  T_{r}\left(  n\right)
\right)  $.\medskip

This statement can be generalized considerably. Indeed, given the integers
$r\geq2$ and $t\geq2,$ let $K_{r}^{+}\left(  t\right)  $ be the complete
$r$-partite graph with each part of size $t,$ and with an edge added to its
first part. The study of $K_{r}^{+}\left(  t\right)  $ in connection to the
Tur\'{a}n theorem has been initiated by Erd\H{o}s \cite{Erd63}, \cite{ErSi73},
but a definite result has been obtained only in \cite{Nik10}:\medskip

\textbf{Theorem C }\emph{Let }$r\geq2$\emph{ and }$c\leq c_{0}\left(
r\right)  $ \emph{be a sufficiently small positive number. If }$G$\emph{ is a
graph of sufficiently large order }$n,$\emph{ with }$e\left(  G\right)
>e\left(  T_{r}\left(  n\right)  \right)  $\emph{, then }$G$\emph{ contains a
}$K_{r}^{+}\left(  \left\lfloor c\log n\right\rfloor \right)  .$\medskip

This type of result is indeed a neat generalization of Simonovits's result,
for any $\left(  r+1\right)  $-color critical graph is a subgraph of
$K_{r}^{+}\left(  \left\lfloor c\log n\right\rfloor \right)  $ if $n$ is large
enough. In \cite{Nik09b}\ a similar theorem has been proved also for the
spectral radius $\lambda^{\left(  2\right)  }:$\medskip

\textbf{Theorem D} \emph{Let }$r\geq2$\emph{ and }$c\leq c_{0}\left(
r\right)  $ \emph{be sufficiently small positive number. If }$G$\emph{ is a
graph of sufficiently large order }$n,$\emph{ with }$\lambda^{\left(
2\right)  }\left(  G\right)  >\lambda^{\left(  2\right)  }\left(  T_{r}\left(
n\right)  \right)  $\emph{, then }$G$\emph{ contains a }$K_{r}^{+}\left(
\left\lfloor c\log n\right\rfloor \right)  .$\medskip

We give a common generalization of Theorems C and D in the following theorem.

\begin{theorem}
\label{sE} Let $r,$ $p,$ $c,$ and $n$ satisfy%
\[
r\geq2,\text{ \ \ }p>1,\text{\ \ \ }0<c\leq r^{-\left(  r+8\right)
r}/2,\text{ \ \ \ and \ \ \ }\log n\geq2p/\left(  cp-c\right)  .
\]
If $G$ is a graph of order $n,$ with $\lambda^{\left(  p\right)  }\left(
G\right)  >\lambda^{\left(  p\right)  }\left(  T_{r}\left(  n\right)  \right)
,$ then $G$ contains a $K_{r}^{+}\left(  \left\lfloor c\log n\right\rfloor
\right)  .$
\end{theorem}

Let us emphasize that in Theorem \ref{sE}\ $c$ may depend on $n$, e.g., if $c$
is a slowly decaying function of $n,$ like $c=1/\log\log n,$ the conclusion is
meaningful for sufficiently large $n.$

It should be noted that the authors of \cite{KLM13}, in their Corollary 2,
prove a similar theorem, where instead of $K_{r}^{+}\left(  \left\lfloor c\log
n\right\rfloor \right)  $ they take a fixed $\left(  r+1\right)  $-color
critical subgraph. However, they claim that their statement generalizes
Theorem D as well, which is false, as the order of $K_{r}^{+}\left(
\left\lfloor c\log n\right\rfloor \right)  $ grows with $n.$ In fact the
change from a fixed $\left(  r+1\right)  $-color critical graph to $K_{r}%
^{+}\left(  \left\lfloor c\log n\right\rfloor \right)  $ is a major
difference, requiring a longer proof, with more advanced techniques and more
delicate calculations.$\medskip$

\subsection{An abstract saturation theorem}

The proofs of Theorems \ref{js} and \ref{sE}, and of several stability results
in a forthcoming paper, will be deduced from a fairly general, multiparameter
statement stated as follows.

\begin{theorem}
\label{t4}Let the numbers $p,$ $\gamma,$ $A,$ $R,$ and $n$ satisfy
\[
1<p\leq2,\ \ \text{\ }0<4\gamma<A<1,\text{ }\ \ R\geq0,\text{ \ \ and
\ \ }n>\frac{4\left(  R+1\right)  p}{\gamma\left(  p-1\right)  }A^{-p/\left(
\gamma p-\gamma\right)  }.
\]
If $G$ is a graph of order $n,$ with
\[
\lambda^{\left(  p\right)  }\left(  G\right)  n^{2/p-1}\geq An-R/n\text{
\ \ and \ \ }\delta\left(  G\right)  \leq\left(  A-\gamma\right)  n,
\]
then there exists an induced subgraph $H\subset G\ $of order $k>A^{-p/\left(
Ap-A\right)  }n,$ with
\[
\lambda^{\left(  p\right)  }\left(  H\right)  k^{2/p-1}>Ak\text{ \ \ \ and
\ \ \ }\delta\left(  H\right)  >\left(  A-\gamma\right)  k.
\]

\end{theorem}

This theorem seems overly complicated, but its meaning and usage are
straightforward. It will be applied to prove the existence of certain
subgraphs. The starting point will be some known statement ensuring that if
$G$ is a graph of sufficiently large order $n,$ with
\[
\lambda^{\left(  p\right)  }\left(  G\right)  n^{2/p-1}>An\text{ \ \ \ and
\ \ \ }\delta\left(  G\right)  >\left(  A-\gamma\right)  n,
\]
then $G$ contains a subgraph $F.$

Now, suppose that $G$ is of sufficiently large order $n,$ but $\lambda
^{\left(  p\right)  }\left(  G\right)  n^{2/p-1}\geq An-O\left(  1\right)  $,
and $\delta\left(  G\right)  \leq\left(  A-\gamma\right)  n,$ so the
requirement for the existence of $F$ are not met at all. In this case Theorem
\ref{t4} helps to mend the situation, as it guarantees that there is an
induced subgraph $H\subset G$ of relatively large order $k,$ satisfying
\[
\lambda^{\left(  p\right)  }\left(  H\right)  k^{2/p-1}>Ak\ \ \ \text{and}%
\ \ \ \delta\left(  H\right)  >\left(  A-\gamma\right)  k.
\]
Now, if $n$ is large enough, then $k$ is large enough, and so $F\subset
H\subset G,$ as desired.

Let us note that, in any concrete case, the choice of $\gamma,$ $A$ and $R$ is
determined by the type of the subgraph $F.$\bigskip

In the remaining part of the paper we prove Theorems \ref{Tur}-\ref{t4}.

\section{Proofs}

\subsection{Notation and preliminaries}

In our proofs we shall use a number of classical inequalities: the Power Mean
inequality (PM inequality), the Bernoulli and the Maclaurin inequalities; for
more details on these inequalities we refer the reader to \cite{HLP88}%
.\medskip

For graph notation and concepts undefined here, we refer the reader to
\cite{Bol98}. In particular, given a graph $G,$ we write:\medskip

- $V\left(  G\right)  $ for the vertex set of $G$ and $v\left(  G\right)  $
for $\left\vert V\left(  G\right)  \right\vert ;$

- $E\left(  G\right)  $ for the edge set of $G$ and $e\left(  G\right)  $ for
$\left\vert E\left(  G\right)  \right\vert ;$

- $\Gamma_{G}\left(  u\right)  $ for the set of neighbors of a vertex $u$ (we
drop the subscript if $G$ is understood)$;$

- $\delta\left(  G\right)  $ for the minimum degree of $G;$

- $k_{r}\left(  G\right)  $ for the number of $r$-cliques of $G;$

- $G-u$ for the graph obtained by removing the vertex $u\in V\left(  G\right)
.\medskip$

If $G$ is a graph of order $n$ and $V\left(  G\right)  $ is not defined
explicitly, it is assumed that $V\left(  G\right)  :=\left\{  1,\ldots
,n\right\}  .$\bigskip

\subsubsection{Some facts about the $p$-spectral radius}

All required facts about the $p$-spectral radius of graphs are given below.
Additional reference material can be found in \cite{KNY14}, \cite{NikA}, and
\cite{NikB}.

Let $G$ be a graph of order $n.$ A vector $\left[  x_{i}\right]  \in
\mathbb{R}^{n}$ such that $\left\vert x_{1}\right\vert ^{p}+\cdots+\left\vert
x_{n}\right\vert ^{p}=1$ and $\lambda^{\left(  p\right)  }\left(  G\right)
=P_{G}\left(  \left[  x_{i}\right]  \right)  $ is called an \emph{eigenvector
}to $\lambda^{\left(  p\right)  }\left(  G\right)  .$ It is easy to see,that
there is always a non-negative eigenvector to $\lambda^{\left(  p\right)
}\left(  G\right)  .$ If $p>1$, by Lagrange's method, one can show that
\begin{equation}
\lambda^{\left(  p\right)  }\left(  G\right)  x_{k}^{p-1}=\sum_{i\in
\Gamma\left(  k\right)  }x_{i}.\label{eeq}%
\end{equation}
for each $k=1,\ldots,n.$ Equation (\ref{eeq}) is called the
\emph{eigenequation} of $\lambda^{\left(  p\right)  }\left(  G\right)  $ for
the vertex $k.$

In the following three bounds it is assumed that $p\geq1.$ First, by
Maclaurin's and the PM inequalities we find the absolute maximum of
$\lambda^{\left(  p\right)  }\left(  G\right)  $ with respect to $n:$
\begin{equation}
\lambda^{\left(  p\right)  }\left(  G\right)  \leq2\sum_{1\leq i<j\leq n}%
x_{i}x_{j}<\left(  n-1\right)  n\left(  \frac{1}{n}\sum_{i=1}^{n}x_{i}\right)
^{2}\leq\left(  n-1\right)  n\left(  \frac{1}{n}\sum_{i=1}^{n}x_{i}%
^{p}\right)  ^{2/p}=\frac{n-1}{n^{2/p-1}}.\label{max}%
\end{equation}
Second, we find a bound with respect to $e\left(  G\right)  :$
\begin{equation}
\lambda^{\left(  p\right)  }\left(  G\right)  \leq2e\left(  G\right)
^{1-1/p}\left(  \sum_{1\leq i<j\leq n}x_{i}^{p}x_{j}^{p}\right)  ^{1/p}%
\leq2e\left(  G\right)  ^{1-1/p}\left(  \frac{n-1}{2n}\right)  ^{1/p}%
\leq\left(  2e\left(  G\right)  \right)  ^{1-1/p}.\label{me}%
\end{equation}
In the other direction, taking the $n$-vector $\mathbf{x}=\left(
n^{-1/p},\ldots,n^{-1/p}\right)  ,$ we obtain a useful lower bound%
\begin{equation}
\lambda^{\left(  p\right)  }\left(  G\right)  \geq P_{G}\left(  \mathbf{x}%
\right)  =2e\left(  G\right)  n^{-2/p}.\label{md}%
\end{equation}
Note that if $1\leq p<2,$ then bound (\ref{md}) may not be tight for some
regular graphs, but for $p\geq2$ it is always tight for regular graphs; in
fact, as mentioned earlier,%
\[
\lim_{p\rightarrow\infty}\lambda^{\left(  p\right)  }\left(  G\right)
n^{2/p}=\lim_{p\rightarrow\infty}\lambda^{\left(  p\right)  }\left(  G\right)
=2e\left(  G\right)  .
\]

It is worth noting that using the PM inequality, one can find that
$\lambda^{\left(  p\right)  }\left(  G\right)  n^{2/p}$ is nonincreasing in
$p,$ that is to say, if $p>q\geq1,$ then
\begin{equation}
\lambda^{\left(  q\right)  }\left(  G\right)  n^{2/q}\geq\lambda^{\left(
p\right)  }\left(  G\right)  n^{2/p}. \label{inc}%
\end{equation}

\subsection{Proof of Theorem \ref{Tur}}

Since a statement similar to Theorem \ref{Tur} has been claimed in
\cite{KLM13}, Corollary 2, we need to make a comment here. The proof given
below reduces Theorem \ref{Tur} to $r$-partite graphs, for which we already
gave an independent proof in \cite{KNY14}. The same reduction, albeit more
complicated, has been carried out in \cite{KLM13} as well, but these authors
provide no proof for $r$-partite graphs, so their proof of Theorem \ref{Tur}
is essentially incomplete. Unfortunately, this omission is not negligible, as
the proof for $r$-partite graphs is much longer and more involved than the
reduction of Theorem \ref{Tur} to $r$-partite graphs.

Next, we state the main ingredient of our proof, which is a particular
instance of a result in \cite{KNY14} about the $p$-spectral radius of
$k$-partite uniform hypergraphs.

\begin{theorem}
\label{th1}Let $r\geq2,$ and $p>1.$ If $G\ $is an $r$-partite graph of order
$n,$ then
\[
\lambda^{\left(  p\right)  }\left(  G\right)  <\lambda^{\left(  p\right)
}\left(  T_{r}\left(  n\right)  \right)  ,
\]
unless $G=T_{r}\left(  n\right)  .$
\end{theorem}

Thus, to prove Theorem \ref{Tur}, all we need is that the maximum
$\lambda^{\left(  p\right)  }\left(  G\right)  $ of a $K_{r+1}$-free graph $G$
of order $n$ is attained on an $r$-partite graph. Reductions of this kind have
been pioneered by Zykov \cite{Zyk49} and Erd\H{o}s, but to spectral problems
they have been first applied by Guiduli, in an unpublished proof of the
spectral Tur\'{a}n theorem. Another noteworthy application of the same
techniques is for the spectral radius of the signless Laplacian of $K_{r+1}%
$-free graphs in \cite{HJZ11}. Thus we proceed with a reduction lemma for
$\lambda^{\left(  p\right)  }\left(  G\right)  $ of a $K_{r+1}$-free graph $G$.

\begin{lemma}
Let $p\geq1.$ If $G$ is a $K_{r+1}$-free graph of order $n,$ then there exists
an $r$-partite graph $H$ of order $n$ such that $\lambda^{\left(  p\right)
}\left(  H\right)  \geq\lambda^{\left(  p\right)  }\left(  G\right)  .$
\end{lemma}

\begin{proof}
Let $\mathbf{x}=\left[  x_{i}\right]  $ be a nonnegative eigenvector to
$\lambda^{\left(  p\right)  }\left(  G\right)  .$ For each $v\in V\left(
G\right)  ,$ set
\[
D_{G}\left(  v,\mathbf{x}\right)  :=%
{\displaystyle\sum\limits_{i\in\Gamma_{G}\left(  v\right)  }}
x_{i}.
\]
We shall prove that there exists a complete $r$-partite graph $H$ such that
$V\left(  H\right)  =V\left(  G\right)  $ and $D_{H}\left(  v,\mathbf{x}%
\right)  \geq D_{G}\left(  v,\mathbf{x}\right)  $ for any $v\in V\left(
G\right)  .$ This proof will be carried out by induction on $r$. Let $u\in
V\left(  G\right)  $ satisfy
\[
D_{G}\left(  u,\mathbf{x}\right)  :=\max\left\{  D_{G}\left(  v,\mathbf{x}%
\right)  :v\in V\left(  G\right)  \right\}  ,
\]
and set $U:=\Gamma_{G}\left(  u\right)  $ and $W:=V\left(  G\right)
\backslash\Gamma_{G}\left(  u\right)  .$ To start the induction let $r:=2;$
hence $G$ is triangle-free, and so $e\left(  G\left[  U\right]  \right)  =0.$
We shall show that the complete bipartite graph $H$ with bipartition $V\left(
H\right)  =U\cup W$ is as required. Indeed, if $v\in U,$ then $\Gamma
_{G}\left(  v\right)  \subset W,$ and so
\[
D_{H}\left(  v,\mathbf{x}\right)  =%
{\displaystyle\sum\limits_{i\in W}}
x_{i}\geq%
{\displaystyle\sum\limits_{i\in\Gamma_{G}\left(  v\right)  }}
x_{i}=D_{G}\left(  v,\mathbf{x}\right)  .
\]
On the other hand, if $v\in W,$ then $D_{H}\left(  v,\mathbf{x}\right)
=D_{G}\left(  u,\mathbf{x}\right)  \geq D_{G}\left(  v,\mathbf{x}\right)  .$
Hence the graph $H$ is as required.

Now, let $r>2$ and assume that the assertion is true for $r^{\prime}$ whenever
$2\leq r^{\prime}<r$. First note that $G\left[  U\right]  $ is a $K_{r}$-free
graph; hence, by the induction assumption there exists a complete $\left(
r-1\right)  $-partite graph $F$ with $V\left(  F\right)  =U$ and $D_{F}\left(
v,\mathbf{x}\right)  \geq D_{G\left[  U\right]  }\left(  v,\mathbf{x}\right)
$ for any vertex $v\in U.$ Let $V\left(  F\right)  =V_{1}\cup\cdots\cup
V_{r-1}$ be the partition of $V\left(  F\right)  $ into independent sets and
let $H$ be the complete $r$-partite graph with partition
\[
V\left(  H\right)  =V_{1}\cup\cdots\cup V_{r-1}\cup W=V\left(  G\right)  .
\]
We shall prove that $H$ is as required. Indeed, on the one hand, if $v\in U,$
then
\[
D_{H}\left(  v,\mathbf{x}\right)  =D_{F}\left(  v,\mathbf{x}\right)  +%
{\displaystyle\sum\limits_{i\in W}}
x_{i}\geq D_{G\left[  U\right]  }\left(  v,\mathbf{x}\right)  +%
{\displaystyle\sum\limits_{i\in\Gamma_{G}\left(  v\right)  \cap W}}
x_{i}=D_{G}\left(  v,\mathbf{x}\right)  .
\]
On the other hand, if $v\in W,$ then $D_{H}\left(  v,\mathbf{x}\right)
=D_{G}\left(  u,\mathbf{x}\right)  \geq D_{G}\left(  v,\mathbf{x}\right)  .$
Hence, $H$ is a complete $r$-partite graph such that $D_{H}\left(  v\right)
\geq D_{G}\left(  v\right)  $ for any $v\in V\left(  G\right)  .$ This
completes the induction step, and the existence of $H$ is proved$.$

To finish the proof of the lemma, note that
\begin{align*}
\lambda^{\left(  p\right)  }\left(  H\right)   &  \geq2%
{\displaystyle\sum\limits_{\left\{  i,j\right\}  \in E\left(  H\right)  }}
x_{i}x_{j}=%
{\displaystyle\sum\limits_{i\in V\left(  H\right)  }}
x_{i}D_{H}\left(  i,\mathbf{x}\right)  \geq%
{\displaystyle\sum\limits_{i\in V\left(  H\right)  }}
x_{i}D_{G}\left(  i,\mathbf{x}\right)  =2%
{\displaystyle\sum\limits_{\left\{  i,j\right\}  \in E\left(  G\right)  }}
x_{i}x_{j}\\
&  =\lambda^{\left(  p\right)  }\left(  G\right)  .
\end{align*}

\end{proof}

\bigskip

\subsection{Proofs of Theorems \ref{Turc} and \ref{Turc1}}

We use below the result of Motzkin and Straus \cite{MoSt65}, that can be
stated as:\medskip

\textbf{Theorem E}\emph{ If }$G$\emph{ is a }$K_{r+1}$\emph{-free graph of
order }$n,$\emph{ and }$x_{1},\ldots,x_{n}$\emph{ are nonnegative numbers such
that }$x_{1}+\cdots+x_{n}=1,$\emph{ then}%
\begin{equation}
2\sum_{\left\{  i,j\right\}  \in E\left(  G\right)  }x_{i}x_{j}\leq1-\frac
{1}{r}. \label{MS}%
\end{equation}

The conditions for equality in (\ref{MS}) are well known, but we shall omit
them. Instead we just note that if $K_{r}\subset G,$ one may choose a vector
$\left(  x_{1},\ldots,x_{n}\right)  $ so that equality holds in (\ref{MS}).

We often shall use the following bound on the number of edges of the Tur\'{a}n
graph $T_{r}\left(  n\right)  ,$%
\begin{equation}
2e\left(  T_{r}\left(  n\right)  \right)  \geq\left(  1-\frac{1}{r}\right)
n^{2}-\frac{r}{4}. \label{te}%
\end{equation}
Indeed, let $n=rs+t,$ where $s$ and $t$ are nonnegative integers and $0\leq
t<s.$ It is known that
\[
e\left(  T_{r}\left(  n\right)  \right)  =\binom{r}{2}\frac{n^{2}-t^{2}}%
{r^{2}}+\binom{t}{2};
\]
hence,
\[
2e\left(  T_{r}\left(  n\right)  \right)  =2\binom{r}{2}\frac{n^{2}-t^{2}%
}{r^{2}}+2\binom{t}{2}=\left(  1-\frac{1}{r}\right)  n^{2}-\frac{t\left(
r-t\right)  }{r}\geq\left(  1-\frac{1}{r}\right)  n^{2}-\frac{r}{4}.
\]
\medskip

\begin{proof}
[\textbf{Proof of Theorem \ref{Turc}}]The proof of inequality (\ref{in0}) has
been given many times, but it is short, so for reader's sake we shall give it
again. Let $\left[  x_{i}\right]  $ be a nonnegative eigenvector to
$\lambda^{\left(  p\right)  }\left(  G\right)  .$ The PM inequality implies
that
\begin{equation}
\lambda^{\left(  p\right)  }\left(  G\right)  =2\sum_{\left\{  i,j\right\}
\in E\left(  G\right)  }x_{i}x_{j}\leq2e\left(  G\right)  ^{1-1/p}\left(
\sum_{\left\{  i,j\right\}  \in E\left(  G\right)  }x_{i}^{p}x_{j}^{p}\right)
^{1/p}. \label{in2}%
\end{equation}
Note that $x_{1}^{p}+\cdots+x_{n}^{p}=1,$ and $G$ is $K_{r+1}$-free, so the
Motzkin-Straus result implies that
\[
\sum_{\left\{  i,j\right\}  \in E\left(  G\right)  }x_{i}^{p}x_{j}^{p}%
\leq\frac{r-1}{2r}.
\]
Plugging this back in (\ref{in2}), we obtain (\ref{in0}).

To prove (\ref{in1}), we use (\ref{in0}) and the concise Tur\'{a}n theorem,
which implies that
\[
2e\left(  G\right)  \leq\left(  1-\frac{1}{r}\right)  n^{2}.
\]
Now, if equality holds, i.e., if
\[
\lambda^{\left(  p\right)  }\left(  G\right)  =\left(  1-\frac{1}{r}\right)
n^{2-2/p},
\]
then we should have
\[
2e\left(  G\right)  =\left(  1-\frac{1}{r}\right)  n^{2},
\]
and this can happen only if $r|n$ and $G=T_{r}\left(  n\right)  .$
\end{proof}

\bigskip

\begin{proof}
[\textbf{Proof of Theorem \ref{Turc1}}]Equality (\ref{l1}) follows from the
Motzkin-Straus' Theorem E and the fact the $K_{r}\subset T_{r}\left(
n\right)  .$ The lower bound in (\ref{le}) follows by (\ref{md}). To prove the
upper bound in (\ref{le}) let us start with
\[
\lambda^{\left(  p\right)  }\left(  T_{r}\left(  n\right)  \right)
\leq\left(  1-\frac{1}{r}\right)  ^{1/p}\left(  2e\left(  T_{r}\left(
n\right)  \right)  \right)  ^{1-1/p}.
\]
Next, using (\ref{te}) and Bernoulli's inequality, we get the estimate
\[
2e\left(  T_{r}\left(  n\right)  \right)  >\left(  1-\frac{1}{r}\right)
n^{2}-\frac{r}{4}>\frac{\left(  1-\frac{1}{r}\right)  n^{2}}{1+\frac{r}{n^{2}%
}}>\frac{\left(  1-\frac{1}{r}\right)  n^{2}}{\left(  1+\frac{r}{pn^{2}%
}\right)  ^{p}}.
\]
This implies%
\[
\frac{\left(  2e\left(  T_{r}\left(  n\right)  \right)  \right)  ^{1/p}%
}{n^{2/p}}\left(  1+\frac{r}{pn^{2}}\right)  >\left(  1-\frac{1}{r}\right)
^{1/p},
\]
and the upper bound in (\ref{le}) follows.

The bound in (\ref{lv}) comes from (\ref{md}) and (\ref{me}).
\end{proof}

\bigskip

\subsection{Proof of Theorem \ref{t4}}

The proof of Theorem \ref{t4} goes along lines, which are familiar from
Theorem 5 in \cite{Nik08}, but the arguments and calculations are more
complicated. To clarify the structure of the proof we have extracted two of
its essential points into Lemmas \ref{le0} and \ref{le1}.

\begin{lemma}
\label{le0}Let the numbers $p,$ $A,$ $\gamma,$ $R,$ and $n$ satisfy
\[
1<p\leq2,\ \ \text{\ }0<\gamma<A<1,\text{ }\ \ R\geq0,\text{ \ \ and
\ \ \ }n\geq4R/\gamma.
\]
Let $G$ be a graph of order $n,$ with
\[
\lambda^{\left(  p\right)  }\left(  G\right)  n^{2/p-1}>An-\frac{R}{n}\text{
\ \ \ and \ \ \ }\delta\left(  G\right)  <\left(  A-\gamma\right)  n.
\]
If $\left[  x_{i}\right]  $ is a nonnegative eigenvector to $\lambda^{\left(
p\right)  }\left(  G\right)  ,$ then the value $\sigma:=\min\left\{  x_{1}%
^{p},\ldots,x_{n}^{p}\right\}  $ satisfies%
\begin{equation}
\sigma\leq\frac{1-\gamma/2}{n}. \label{s}%
\end{equation}

\end{lemma}

\begin{proof}
Let $k\in V\left(  G\right)  $ be a vertex of degree $\delta=\delta\left(
G\right)  $ and set for short $\lambda^{\left(  p\right)  }\left(  G\right)
=\lambda$. Applying the PM inequality to the right side of the eigenequation
for $x_{k},$ we get
\begin{align*}
\lambda\sigma^{1-1/p}  &  \leq\lambda x_{k}^{p-1}=\sum_{i\in\Gamma\left(
k\right)  }x_{i}\leq\delta^{1-1/p}\left(  \sum_{i\in\Gamma\left(  k\right)
}x_{i}^{p}\right)  ^{1/p}=\delta^{1-1/p}\left(  1-\sum_{i\notin\Gamma\left(
k\right)  }x_{i}^{p}\right)  ^{1/p}\\
&  \leq\delta^{1-1/p}\left(  1-\left(  n-\delta\right)  \sigma\right)  ^{1/p}.
\end{align*}
After some algebra, this inequality reduces to%
\[
\frac{\lambda^{p}\sigma^{p-1}}{\delta^{p-1}}+\left(  n-\delta\right)
\sigma\leq1.
\]
In view of (\ref{md}) $\lambda n^{2/p-1}\geq2e\left(  G\right)  /n\geq\delta;$
hence
\[
\left(  \frac{\lambda n^{2/p-1}}{\delta}\right)  ^{p-1}\geq1,
\]
and so,
\[
\frac{\lambda^{p-1}}{\delta^{p-1}}>n^{\left(  1-2/p\right)  \left(
p-1\right)  }.
\]
Also $\sigma\leq1/n$ and since $p-2\leq0,$ we see that $\sigma^{p-2}\geq
n^{2-p}.$ Therefore,%
\[
\frac{\lambda^{p}\sigma^{p-1}}{\delta^{p-1}}\geq\lambda n^{\left(
1-2/p\right)  \left(  p-1\right)  -p+2}\sigma=\lambda n^{2/p-1}\sigma,
\]
yielding finally
\[
\lambda n^{2/p-1}\sigma+\left(  n-\delta\right)  \sigma\leq1.
\]
Now, plugging the bounds on $\delta\left(  G\right)  $ and $\lambda^{\left(
p\right)  }\left(  G\right)  n^{2/p-1},$ we get
\[
\sigma\leq\frac{1}{\left(  A-\frac{R}{n}\right)  n+n-\left(  A-\gamma\right)
n}=\frac{1}{\left(  1+\gamma\right)  n-\frac{R}{n}}.
\]
To complete the proof of the lemma we shall check that
\[
\frac{1}{1+\gamma-R/n^{2}}<1-\frac{\gamma}{2}.
\]
Indeed,
\begin{align*}
\left(  1-\frac{\gamma}{2}\right)  \left(  1+\gamma-\frac{R}{n^{2}}\right)
&  =1+\frac{\gamma}{2}\left(  1-\gamma\right)  -\frac{R}{n^{2}}\left(
1-\frac{\gamma}{2}\right)  >1+\frac{\gamma}{2}\left(  1-\gamma\right)
-\frac{R}{n}\left(  1-\frac{\gamma}{2}\right) \\
&  >1+\frac{\gamma}{2}\left(  1-\gamma\right)  -\frac{\gamma}{4}\left(
1-\frac{\gamma}{2}\right)  =1+\frac{\gamma}{2}\left(  \frac{3}{4}-\frac{7}%
{8}\gamma\right) \\
&  >1+\frac{\gamma}{2}\left(  \frac{3}{4}-\frac{7}{16}\right)  >1.
\end{align*}
Lemma \ref{le0} is proved.
\end{proof}

\bigskip

The next lemma shows that if $\lambda^{\left(  p\right)  }\left(  G\right)
n^{2/p-1}$ is large enough, but the minimum degree $\delta\left(  G\right)  $
is not too large, we can remove a vertex $u$, so that $\lambda^{\left(
p\right)  }\left(  G-u\right)  \left(  n-1\right)  ^{2/p-1}$ is also large.

\begin{lemma}
\label{le1}Let the numbers $p,$ $\gamma,$ $A,$ $R,$ and $n$ satisfy
\[
1<p\leq2,\ \ \text{\ }0<\gamma<A<1,\text{ }\ \ R\geq0,\text{ \ \ and
\ \ \ }n\geq4R/\gamma.
\]
Let $G$ be a graph of order $n,$ with
\[
\delta\left(  G\right)  \leq\left(  A-\gamma\right)  n\text{ \ \ and
\ }\lambda^{\left(  p\right)  }\left(  G\right)  n^{2/p-1}\geq An-R/n.
\]
If $\left[  x_{i}\right]  $ is a nonnegative eigenvector to $\lambda^{\left(
p\right)  }\left(  G\right)  $ and $u$ is a vertex with $x_{u}=\min\left\{
x_{1},\ldots,x_{n}\right\}  ,$ then%
\[
\lambda^{\left(  p\right)  }\left(  G-u\right)  \left(  n-1\right)
^{2/p-1}\geq\left(  \frac{n-2}{n-1}\right)  ^{1-\left(  1-1/p\right)  \gamma
}\lambda^{\left(  p\right)  }\left(  G\right)  n^{2/p-1}.
\]

\end{lemma}

\begin{proof}
Let $p,$ $\gamma,$ $A,$ $R,$ and $n$ satisfy the requirements, let
$\mathbf{x}=\left[  x_{i}\right]  $ be a nonnegative eigenvector to
$\lambda^{\left(  p\right)  }\left(  G\right)  $ and $u$ be a vertex with
$x_{u}:=\min\left\{  x_{1},\ldots,x_{n}\right\}  ;$ set $\sigma:=x_{k}^{p}.$
Obviously, Lemma \ref{le0} can be applied here, getting%
\begin{equation}
\sigma\leq\frac{1-\gamma/2}{n}.\label{ls}%
\end{equation}
Next, set for short $\delta:=\delta\left(  G\right)  ,$ $\lambda_{n}%
:=\lambda^{\left(  p\right)  }\left(  G\right)  ,$ and $\lambda_{n-1}%
:=\lambda^{\left(  p\right)  }\left(  G-u\right)  .$ Letting $\mathbf{x}%
^{\prime}$ be the $\left(  n-1\right)  $-vector obtained from $\mathbf{x}$ by
omitting the entry $x_{k},$ we see that
\[
P_{G-k}\left(  \mathbf{x}^{\prime}\right)  =P_{G}\left(  \mathbf{x}\right)
-2x_{k}\sum_{i\in\Gamma\left(  k\right)  }x_{i}=\lambda_{n}-2x_{k}\left(
\lambda_{n}x_{k}^{p-1}\right)  =\lambda_{n}\left(  1-2x_{k}^{p}\right)  .
\]
On the other hand,
\[
P_{G-k}\left(  \mathbf{x}^{\prime}\right)  \leq\lambda_{n-1}\left\Vert
\mathbf{x}^{\prime}\right\Vert _{p}^{2}=\lambda_{n-1}\left(  1-x_{k}%
^{p}\right)  ^{2/p},
\]
hence, after some algebra, we find that
\begin{equation}
\lambda_{n-1}\geq\frac{1-2\sigma}{\left(  1-\sigma\right)  ^{2/p}}\lambda
_{n}.\label{in10}%
\end{equation}
Note that the function
\[
f\left(  x\right)  =\frac{1-2y}{\left(  1-y\right)  ^{2/p}}%
\]
is decreasing in $y$ for $0<y<1,$ for the derivative of $f\left(  y\right)  $
satisfies
\[
f^{\prime}\left(  y\right)  =\frac{-2\left(  1-y\right)  ^{2/p}+\frac{2}%
{p}\left(  1-2y\right)  \left(  1-y\right)  ^{2/p-1}}{\left(  1-y\right)
^{4/p}}=-\frac{2\left(  1-y\right)  ^{2/p-1}}{p\left(  1-y\right)  ^{4/p}%
}\left(  \left(  p-1\right)  +2y\right)  <0.
\]
Therefore, in view of (\ref{ls}), we find that
\[
f\left(  \sigma\right)  \geq f\left(  \frac{1-\gamma/2}{n}\right)  .
\]
Thus, setting for short $\xi:=\gamma/2$, we see that
\[
\frac{1-2\sigma}{\left(  1-\sigma\right)  ^{2/p}}\geq\frac{1-2\frac{1-\xi}{n}%
}{\left(  1-\frac{1-\xi}{n}\right)  ^{2/p}}.
\]
Plugging this back in (\ref{in10}), we find that
\begin{align*}
\frac{\lambda_{n-1}\left(  n-1\right)  ^{2/p-1}}{\lambda_{n}n^{2/p-1}} &
\geq\left(  1-\frac{\left(  1-\xi\right)  /n}{1-\left(  1-\xi\right)
/n}\right)  \cdot\left(  \frac{1-1/n}{1-\left(  1-\xi\right)  /n}\right)
^{2/p-1}\\
&  =\left(  1-\frac{1-\xi}{n-1+\xi}\right)  \cdot\left(  1-\frac{\xi}{n-1+\xi
}\right)  ^{2/p-1}.
\end{align*}
To estimate the latter expression, note that $0<1-\xi<1$ and $0<\xi<1;$ hence,
Bernoulli's inequality implies that
\[
\left(  1-\frac{1-\xi}{n-1+\xi}\right)  \geq\left(  1-\frac{1}{n-1+\xi
}\right)  ^{1-\xi},
\]
and
\[
\left(  1-\frac{\xi}{n-1+\xi}\right)  \geq\left(  1-\frac{1}{n-1+\xi}\right)
^{\xi}.
\]
Thus, we obtain
\[
\frac{\lambda_{n-1}\left(  n-1\right)  ^{2/p-1}}{\lambda_{n}n^{2/p-1}}%
\geq\left(  1-\frac{1}{n-1+\xi}\right)  ^{1-2\left(  1-1/p\right)  \xi
}>\left(  1-\frac{1}{n-1}\right)  ^{1-\left(  1-1/p\right)  \gamma},
\]
as claimed. Lemma \ref{le1} is proved.
\end{proof}

The main idea of the proof of Theorem \ref{t4} is to iterate the removal of
vertices of smallest entry in eigenvectors to $\lambda^{\left(  p\right)  }$.
Every time a vertex is removed, the ratio of $\lambda^{\left(  p\right)  }$ of
the remaining graph to its order increases. So the vertex removal must stop
before $\lambda^{\left(  p\right)  }$ exceeds its absolute maximum. As this
stop happens fairly soon, the order of the remaining graph is fairly
large.\bigskip

\begin{proof}
[\textbf{Proof of Theorem \ref{t4}}]Let $p,\gamma,A,R$ and $n$ be as required,
and let $G$ be a graph of order $n,$ with
\[
\lambda^{\left(  p\right)  }\left(  G\right)  n^{2/p-1}>An-R/n\text{ \ \ and
\ \ }\delta\left(  G\right)  \leq\left(  A-\gamma\right)  n.
\]
Define a decreasing sequence of graphs $G_{n}\supset G_{n-1}\supset\cdots$ by
the following procedure $\mathcal{P}:\medskip$

$\qquad G_{n}:=G;$

\qquad$i:=n;$

\qquad\textbf{while }$\delta\left(  G_{i}\right)  \leq\left(  A-\gamma\right)
i$ \textbf{begin}

1.\qquad\emph{Select an eigenvector }$\left(  x_{1},\ldots,x_{i}\right)  $
\emph{to }$\lambda^{\left(  p\right)  }\left(  G_{i}\right)  ;$

2.\qquad\emph{Select a vertex }$u\in V\left(  G_{i}\right)  $ \emph{with
}$x_{u}=\min\left\{  x_{1},\ldots,x_{i}\right\}  ;$

3.\qquad$G_{i-1}:=G_{i}-u;$

4.$\qquad i:=i-1;$

\qquad\textbf{end}.\textbf{\medskip}

We claim that the following compound statement is true:

\emph{(i) }at line 1 we always have
\begin{equation}
i>A^{p/\left(  \gamma p-\gamma\right)  }n>4R/\gamma,\label{b1}%
\end{equation}

\emph{(ii) }at line 3 we always have
\begin{equation}
\lambda^{\left(  p\right)  }\left(  G_{i-1}\right)  \left(  i-1\right)
^{2/p-1}>\left(  1-\frac{1}{i-1}\right)  ^{1-\left(  1-1/p\right)  \gamma
}\lambda^{\left(  p\right)  }\left(  G_{i}\right)  i^{2/p-1}.\label{b2}%
\end{equation}

Clearly, to prove (\ref{b2}) we may use Lemma \ref{le1}, which, however,
requires that $i>4R/\gamma;$ this is why we have to prove (\ref{b1}) as well.
We shall use induction on $d:=n-i.$ To start the induction let $d=n-n=0.$
Clearly inequality (\ref{b1}) is true for $i=n$. Since, at line 1 we always
have $\delta\left(  G_{i}\right)  \leq\left(  A-\gamma\right)  i,$ after
removing the vertex $u$ Lemma \ref{le1}, together with (\ref{lv}), implies
(\ref{b2}). Now, assume that (\ref{b1}) and (\ref{b2}) hold for $0\leq d\leq
n-i;$ we shall prove them for $d+1=n-\left(  i-1\right)  .$ First, the
inductive assumption implies that
\[
\frac{\lambda^{\left(  p\right)  }\left(  G_{s}\right)  s^{2/p-1}}%
{\lambda^{\left(  p\right)  }\left(  G_{s+1}\right)  \left(  s+1\right)
^{2/p-1}}>\left(  1-\frac{1}{s}\right)  ^{1-\left(  1-1/p\right)  \gamma}%
\]
for each $s=n-1,\ldots,i.$ Hence, multiplying these inequalities for
$s=n-1,\ldots,i,$ we obtain
\[
\lambda^{\left(  p\right)  }\left(  G_{i}\right)  i^{2/p-1}>\left(  \frac
{i-1}{n-1}\right)  ^{1-\left(  1-1/p\right)  \gamma}\lambda^{\left(  p\right)
}\left(  G_{n}\right)  n^{2/p-1}.
\]
On the other hand, by (\ref{max}) we have%
\begin{align}
i-1 &  \geq\lambda^{\left(  p\right)  }\left(  G_{i}\right)  i^{2/p-1}>\left(
\frac{i-1}{n-1}\right)  ^{1-\left(  1-1/p\right)  \gamma}\lambda^{\left(
p\right)  }\left(  G_{n}\right)  n^{2/p-1}\nonumber\\
&  >\left(  \frac{i-1}{n-1}\right)  ^{1-\left(  1-1/p\right)  \gamma}\left(
An-\frac{R}{n}\right)  >\left(  \frac{i-1}{n-1}\right)  ^{1-\left(
1-1/p\right)  \gamma}A\left(  n-1\right)  .\label{x}%
\end{align}
In the last derivation we use that $n\geq R/A,$ which follows from
\[
n>\frac{4R}{\gamma}A^{-p/\left(  \gamma p-\gamma\right)  }>\frac{R}{\gamma
}>RA^{-1}.
\]
From (\ref{x}), we see that%
\[
\left(  \frac{i}{n}\right)  ^{\left(  1-1/p\right)  \gamma}>\left(  \frac
{i-1}{n-1}\right)  ^{\left(  1-1/p\right)  \gamma}=\frac{i-1}{n-1}\left(
\frac{n-1}{i-1}\right)  ^{1-\left(  1-1/p\right)  \gamma}>A,
\]
and so,
\[
i>nA^{p/\left(  \gamma p-\gamma\right)  }\geq\frac{4\left(  R+1\right)
p}{\gamma\left(  p-1\right)  }A^{-p/\left(  \gamma p-\gamma\right)
}A^{p/\left(  \gamma p-\gamma\right)  }>\frac{4R}{\gamma},
\]
implying (\ref{b1}). Therefore, after removing the vertex $u,$ Lemma
\ref{le1}, together with (\ref{lv}), implies that (\ref{b2}) holds as well.
This completes the induction step and the proof of \emph{(i) }and \emph{(ii)}.

Finally, let $H:=G_{i}$ and $k:=v\left(  H\right)  =i,$ where $G_{i}\ $is the
last graph generated by $\mathcal{P}.$ We shall prove the following three
properties of $H:$%
\begin{align}
\delta\left(  H\right)   &  >\left(  A-\gamma\right)  k,\label{c1}\\
k &  >A^{p/\left(  \gamma p-\gamma\right)  }n\label{c2}\\
\lambda^{\left(  p\right)  }\left(  H\right)  k^{2/p-1} &  >Ak.\label{c3}%
\end{align}
Indeed, inequality (\ref{c1}) is obvious, as this is the loop exit condition.
Also inequality (\ref{c2}) holds because of (\ref{b1}). Finally, note that
\begin{align*}
\lambda^{\left(  p\right)  }\left(  H\right)  k^{2/p-1} &  >\left(  \frac
{k-1}{n-1}\right)  ^{1-\left(  1-1/p\right)  \gamma}\lambda^{\left(  p\right)
}\left(  G_{n}\right)  n^{2/p-1}\\
&  >A\left(  \frac{k-1}{n-1}\right)  ^{1-\left(  1-1/p\right)  \gamma}\left(
n-\frac{R}{nA}\right)  .
\end{align*}
To prove (\ref{c3}), we shall show that%
\begin{equation}
\left(  \frac{k-1}{n-1}\right)  ^{1-\left(  1-1/p\right)  \gamma}\left(
n-\frac{R}{nA}\right)  >k,\label{c4}%
\end{equation}
which is equivalent to
\begin{equation}
1-\frac{R}{n^{2}A}>\frac{k}{n}\left(  \frac{n-1}{k-1}\right)  ^{1-\left(
1-1/p\right)  \gamma}.\label{co5}%
\end{equation}
Assume the latter inequality fails, that is to say,%
\[
1-\frac{R}{n^{2}A}\leq\frac{k}{n}\left(  \frac{n-1}{k-1}\right)  ^{1-\left(
1-1/p\right)  \gamma}=\frac{k\left(  n-1\right)  }{n\left(  k-1\right)
}\left(  \frac{k-1}{n-1}\right)  ^{\left(  1-1/p\right)  \gamma}.
\]
Using Bernoulli's inequality, we get
\begin{align*}
\left(  1+\frac{n-k}{n\left(  k-1\right)  }\right)  \left(  \frac{k-1}%
{n-1}\right)  ^{\left(  1-1/p\right)  \gamma} &  \leq\left(  1+\frac
{n-k}{n\left(  k-1\right)  }\right)  \left(  1-\left(  1-1/p\right)
\gamma\frac{n-k}{n-1}\right)  \\
&  <1+\frac{n-k}{n\left(  k-1\right)  }-\left(  1-1/p\right)  \gamma\frac
{n-k}{n-1}.
\end{align*}
After some rearrangement we obtain%
\begin{equation}
\frac{\left(  p-1\right)  }{p}\gamma<\frac{n-1}{n\left(  k-1\right)  }%
+\frac{R\left(  n-1\right)  }{\left(  n-k\right)  n^{2}A}<\frac{1}{\left(
k-1\right)  }+\frac{R}{\left(  n-k\right)  nA}.\label{co1}%
\end{equation}
Now, obviously%
\[
\frac{1}{k-1}<\left(  nA^{p/\left(  \gamma p-\gamma\right)  }-1\right)
^{-1}<\left(  \frac{4\left(  R+1\right)  p}{\gamma\left(  p-1\right)
}-1\right)  ^{-1}<\left(  \frac{2p}{\gamma\left(  p-1\right)  }\right)
^{-1}=\frac{\gamma\left(  p-1\right)  }{2p},
\]
and also%
\begin{align*}
\frac{R}{\left(  n-k\right)  nA} &  <\frac{R}{An}<\frac{R}{A}\cdot\frac
{\gamma\left(  p-1\right)  }{4\left(  R+1\right)  p}\cdot A^{p/\left(  \gamma
p-\gamma\right)  }<\frac{1}{A}\cdot\frac{\gamma\left(  p-1\right)  }%
{4p}A^{1/\gamma}\\
&  <\frac{1}{A}\cdot\frac{\gamma\left(  p-1\right)  }{2p}A^{2}<\frac
{\gamma\left(  p-1\right)  }{2p}.
\end{align*}
Therefore, (\ref{co1}) is a contradiction and (\ref{c4}) holds.

Hence the graph $H$ has the required properties and Theorem \ref{t4} is proved.
\end{proof}

\subsection{Proof of Theorem \ref{js}}

The proof of Theorem \ref{js} is based on the following nonspectral result,
proved in \cite{BoNi04}.

\begin{lemma}
\label{leKd} Let $r\geq2$ and $G$ be graph a of order $n.$ If $G$ contains a
$K_{r+1}$ and $\ \delta\left(  G\right)  >\left(  1-1/r-1/r^{4}\right)  n,$
then $\mathrm{js}_{r+1}\left(  G\right)  >n^{r-1}/r^{r+3}.\hfill\square$
\end{lemma}

The idea of Lemma \ref{leKd} can be traced back to Erd\H{o}s; its main
advantage is that the bound on the jointsize can be deduced from two simpler
conditions: presence of $K_{r+1}$ and sufficient minimum degree$.$ Although,
these conditions may not hold in $G,$ Theorem \ref{t4} guarantees that there
is a large subgraph $H$ of $G$ for which the conditions do hold. Now, applying
Lemma \ref{leKd} to $H$, we obtain the desired bound on $\mathrm{js}%
_{r+1}\left(  G\right)  $.\bigskip

\begin{proof}
[\textbf{Proof of Theorem \ref{js}}]Let $G$ be a graph of order $n$ such that
$\lambda^{\left(  p\right)  }\left(  G\right)  \geq\lambda^{\left(  p\right)
}\left(  T_{r}\left(  n\right)  \right)  $ and assume that $G\neq T_{r}\left(
n\right)  .$ Theorem \ref{Tur} implies that $G$ contains a $K_{r+1}.$ Now, if
\begin{equation}
\delta\left(  G\right)  >\left(  1-1/r-1/r^{4}\right)  n,\label{mnd}%
\end{equation}
then Lemma \ref{leKd} implies that
\[
\mathrm{js}_{r+1}\left(  G\right)  >\frac{n^{r-1}}{r^{r+3}}>\frac{n^{r-1}%
}{r^{r^{6}p/\left(  p-1\right)  }},
\]
completing the proof. Thus we shall assume that (\ref{mnd}) fails. Then,
letting
\[
\gamma:=1/r^{4},\text{ \ \ }A:=1-1/r,\text{ \ \ }R:=r/4,
\]
we see that $\delta\left(  G\right)  \leq\left(  A-\gamma\right)  n,$ and, in
view of (\ref{lv}), we also see that
\[
\lambda^{\left(  p\right)  }\left(  G\right)  n^{2/p-1}\geq2e\left(
T_{r}\left(  n\right)  \right)  /n\geq\left(  1-1/r\right)  n-r/4n=An-R/n.
\]
We want to apply Theorem \ref{t4}, but to do so we have to ensure that
$1<p\leq2$ and that
\begin{equation}
n>\frac{4\left(  R+1\right)  p}{\gamma\left(  p-1\right)  }A^{-p/\left(
\gamma p-\gamma\right)  }=\frac{\left(  r+4\right)  r^{4}p}{p-1}\left(
\frac{r}{r-1}\right)  ^{r^{4}p/\left(  p-1\right)  }.\label{mnn}%
\end{equation}
First we shall show that if (\ref{mnn}) fails, then the proof is trivially
completed. Assume that (\ref{mnn}) fails. Since $K_{r+1}\subset G,$ we have
$\mathrm{js}_{r+1}\left(  G\right)  \geq1$; hence the proof would be
completed, if we can show that
\begin{equation}
1>\frac{n^{r-1}}{r^{r^{6}p/\left(  p-1\right)  }}.\label{mnj}%
\end{equation}
Assume for a contradiction that (\ref{mnj}) fails. Then
\[
r^{r^{6}p/\left(  p-1\right)  }<n^{r-1}<\left(  \frac{\left(  r+4\right)
r^{4}p}{p-1}\right)  ^{r-1}\left(  \frac{r}{r-1}\right)  ^{r^{4}\left(
r-1\right)  p/\left(  p-1\right)  }.
\]
To simplify the right side, we use the obvious inequalities $r/\left(
r-1\right)  <r$ and
\[
\left(  \frac{\left(  r+4\right)  r^{4}p}{p-1}\right)  ^{r-1}<\left(
\frac{4r^{5}p}{p-1}\right)  ^{r}<r^{5r}\left(  \frac{4p}{p-1}\right)
^{r},\text{ }%
\]
thus getting
\[
r^{r^{6}p/\left(  p-1\right)  }<\left(  \frac{4p}{p-1}\right)  ^{r}%
r^{r^{4}\left(  r-1\right)  p/\left(  p-1\right)  +5r}.
\]
Since $r\geq2,$ and $4^{x}>e^{x}>x$ for $x>0,$ we see that%
\[
r^{8rp/\left(  p-1\right)  }>4^{4rp/\left(  p-1\right)  }>\left(  4p/\left(
p-1\right)  \right)  ^{r}.
\]
Hence,%
\begin{align*}
r^{6}p/\left(  p-1\right)   &  <r^{5}\left(  r-1\right)  p/\left(  p-1\right)
+5r+8rp/\left(  p-1\right)  \\
&  =\left(  r^{6}-r^{5}+5r\left(  p-1\right)  /p+8r\right)  p/\left(
p-1\right)  \\
&  <\left(  r^{6}-r^{5}+13r\right)  p/\left(  p-1\right)  ,
\end{align*}
and after obvious cancellations, we find that $r^{5}<13r,$ which is the
desired contradiction. Therefore, we can assume that (\ref{mnn}) holds.

Now, suppose that $1<p\leq2.$ All parameter conditions of Theorem \ref{t4} are
met, and so there is an induced subgraph $H\subset G$ of order
\[
k>A^{p/\left(  \gamma p-\gamma\right)  }n=\left(  1-1/r\right)  ^{r^{4}%
p/\left(  p-1\right)  }n>\frac{n}{r^{r^{4}p/\left(  p-1\right)  }}%
\]
such that $\lambda^{\left(  p\right)  }\left(  H\right)  >\left(
1-1/r\right)  k$ and $\delta\left(  H\right)  >\left(  1-1/r-1/r^{4}\right)
k.$ By Theorem \ref{Tur}, $K_{r+1}\subset H;$ hence, Lemma \ref{leKd} implies
that
\[
\mathrm{js}_{r+1}\left(  H\right)  >\frac{k^{r-1}}{r^{r+3}}>\left(  \frac
{n}{r^{r^{4}p/\left(  p-1\right)  }}\right)  ^{r-1}\frac{1}{r^{r+3}}%
>\frac{n^{r-1}}{r^{r^{4}\left(  r-1\right)  p/\left(  p-1\right)  +r+3}}%
>\frac{n^{r-1}}{r^{r^{6}p/\left(  p-1\right)  }}.
\]
Since $\mathrm{js}_{r+1}\left(  G\right)  \geq\mathrm{js}_{r+1}\left(
H\right)  ,$ the proof of Theorem \ref{js} is completed whenever $1<p\leq2.$

To finish the proof for all $p$, assume that $p>2.$ Then in view of
(\ref{inc}) and (\ref{lv}) we have
\[
\lambda^{\left(  2\right)  }\left(  G\right)  \geq\lambda^{\left(  p\right)
}\left(  G\right)  n^{2/p-1}\geq\lambda^{\left(  p\right)  }\left(
T_{r}\left(  n\right)  \right)  n^{2/p-1}>\left(  1-\frac{1}{r}\right)
n-\frac{r}{4n}.
\]
Applying Theorem \ref{t4} with $p=2,$ we get a subgraph $H\subset G$ of order
\[
k>A^{2r^{4}}n=\left(  1-1/r\right)  ^{2r^{4}}n>\frac{n}{r^{2r^{4}}}%
\]
such that $\lambda^{\left(  p\right)  }\left(  H\right)  >\left(
1-1/r\right)  k$ and $\delta\left(  H\right)  >\left(  1-1/r-1/r^{4}\right)
k.$ By Theorem \ref{Tur}, $K_{r+1}\subset H;$ hence, Lemma \ref{leKd} implies
that
\[
\mathrm{js}_{r+1}\left(  H\right)  >\frac{k^{r-1}}{r^{r+3}}>\left(  \frac
{n}{r^{2r^{4}}}\right)  ^{r-1}\frac{1}{r^{r+3}}>\frac{n^{r-1}}{r^{2r^{4}%
\left(  r-1\right)  +r+3}}>\frac{n^{r-1}}{r^{2r^{5}}}>\frac{n^{r-1}}%
{r^{r^{6}p/\left(  p-1\right)  }}.
\]
Since $\mathrm{js}_{r+1}\left(  G\right)  \geq\mathrm{js}_{r+1}\left(
H\right)  ,$ Theorem \ref{js} is proved completely.\ 
\end{proof}

\bigskip

\subsection{Proof of Theorem \ref{sE}}

For the proof of Theorem \ref{sE} we rely on Theorem \ref{t4} and on a
non-spectral result, proved in \cite{Nik10}, Theorem 6. To state it we first
extend the definition of $K_{r}^{+}\left(  t\right)  $ as follows: given the
integers $r\geq2,$ $s\geq2,$ and $t\geq1,$ let $K_{r}^{+}\left(  s;t\right)  $
be the complete $r$-partite graph with first $r-1$ parts of size $s$ and the
last part of size $t,$ and with an edge added to its first part.

\begin{theorem}
\label{thv4}Let $r,$ $c,$ and $n$ satisfy%
\[
r\geq2,\text{ \ \ \ }0<c\leq r^{-\left(  r+8\right)  r},\text{ \ \ \ and
\ \ \ }\log n\geq2/c.
\]
If $G$ is a graph of order $n,$ with $K_{r+1}\subset G$ and $\delta\left(
G\right)  >\left(  1-1/r-1/r^{4}\right)  n,$ then $G$ contains a $K_{r}%
^{+}\left(  \left\lfloor c\log n\right\rfloor ;\left\lceil n^{1-cr^{3}%
}\right\rceil \right)  .\hfill\square.$
\end{theorem}

For the proof of Theorem \ref{sE} we shall need a corollary of this statement.

\begin{corollary}
\label{cor1}Let $r,$ $a,$ and $n$ satisfy%
\[
r\geq2,\text{ \ \ \ }0<a\leq r^{-\left(  r+8\right)  r},\text{ \ \ \ and
\ \ \ }\log n\geq2/a.
\]
If $G$ is a graph of order $n,$ with $K_{r+1}\subset G$ and $\delta\left(
G\right)  >\left(  1-1/r-1/r^{4}\right)  n,$ then \ $G$ contains a $K_{r}%
^{+}\left(  \left\lfloor a\log n\right\rfloor \right)  .$
\end{corollary}

\begin{proof}
There is not much to prove here. Indeed assume that $r,a,$ and $n$ are as
required and let $G$ be a graph of order $n,$ with $K_{r+1}\subset G$ and
$\delta\left(  G\right)  >\left(  1-1/r-1/r^{4}\right)  n.$ By Theorem
\ref{thv4}, $G$ contains a $K_{r}^{+}\left(  \left\lfloor a\log n\right\rfloor
;\left\lceil n^{1-ar^{3}}\right\rceil \right)  .$ First, note that
\[
\log n\geq\frac{2}{a}\geq2r^{\left(  r+8\right)  r}>2;
\]
hence $n^{1/2}>e$ and so, $n^{1/2}>\log n^{1/2}.$ Now, with a lot to spare, we
see that
\[
n^{1-ar^{3}}>n^{1/2}>\frac{1}{2}\log n>r^{-\left(  r+8\right)  r}\log n\geq
a\log n.
\]
So
\[
K_{r}^{+}\left(  \left\lfloor a\log n\right\rfloor \right)  \subset K_{r}%
^{+}\left(  \left\lfloor a\log n\right\rfloor ;\left\lceil n^{1-ar^{3}%
}\right\rceil \right)  \subset G,
\]
completing the proof of Corollary \ref{cor1}.
\end{proof}

\bigskip

\begin{proof}
[\textbf{Proof of Theorem \ref{sE}}]Let $r,$ $p,$ $c,$ and $n$ be as required,
and let $G$ be a graph of order $n$ with $\lambda^{\left(  p\right)  }\left(
G\right)  >\lambda^{\left(  p\right)  }\left(  T_{r}\left(  n\right)  \right)
;$ thus, by Theorem \ref{Tur}, $G$ contains a $K_{r+1}.$ If%
\begin{equation}
\delta\left(  G\right)  >\left(  1-1/r-1/r^{4}\right)  n,\label{mind1}%
\end{equation}
then Corollary \ref{cor1} implies that $G$ contains a $K_{r}^{+}\left(
\left\lfloor c\log n\right\rfloor \right)  ,$ completing the proof. Thus, we
shall assume that (\ref{mind1}) fails. Then, letting
\[
\gamma:=1/r^{4},\text{\ \ \ }A:=1-1/r,\text{ \ \ }R:=r/4,
\]
we see that $\delta\left(  G\right)  \leq\left(  A-\gamma\right)  n,$ and, in
view of (\ref{lv}), we also see that
\[
\lambda^{\left(  p\right)  }\left(  G\right)  n^{2/p-1}\geq2e\left(
T_{r}\left(  n\right)  \right)  /n\geq\left(  1-1/r\right)  n-r/4n=An-R/n.
\]
To apply Theorem \ref{t4}, we have to ensure that $1<p\leq2$ and that
\begin{equation}
n>\frac{4\left(  R+1\right)  p}{\gamma\left(  p-1\right)  }A^{-p/\left(
\gamma p-\gamma\right)  }=\frac{r^{4}\left(  r+4\right)  p}{p-1}\left(
\frac{r}{r-1}\right)  ^{r^{4}p/\left(  p-1\right)  }.\label{mnn1}%
\end{equation}

First from $\log n>2p/\left(  cp-c\right)  $ we obtain
\[
\log n\geq\frac{2p}{c\left(  p-1\right)  }\geq2r^{\left(  r+8\right)  r}%
\frac{p}{p-1}\geq2r^{\left(  r+8\right)  r}\frac{p}{p-1}\frac{\log r}{r^{2}%
}>r^{5}\frac{p}{p-1}\log r.
\]
On the other hand,
\[
\frac{r^{4}\left(  r+4\right)  p}{p-1}\left(  \frac{r}{r-1}\right)
^{r^{4}p/\left(  p-1\right)  }<r^{5}\left(  \frac{4p}{p-1}\right)
r^{r^{4}p/\left(  p-1\right)  }<r^{r^{4}p/\left(  p-1\right)  +5+8p/\left(
p-1\right)  }<r^{r^{5}p/\left(  p-1\right)  }.
\]
Thus, we see that (\ref{mnn1}) holds.

Now, suppose that $1<p\leq2.$ All parameter conditions of Theorem \ref{t4} are
met, and so there is an induced subgraph $H\subset G$ of order
\[
k>A^{p/\left(  \gamma p-\gamma\right)  }n=\left(  1-1/r\right)  ^{r^{4}%
p/\left(  p-1\right)  }n>\frac{n}{r^{r^{4}p/\left(  p-1\right)  }},
\]
with $\lambda^{\left(  p\right)  }\left(  H\right)  >\left(  1-1/r\right)  k$
and $\delta\left(  H\right)  >\left(  1-1/r-1/r^{4}\right)  k.$ We shall prove
that $\log k>2r^{\left(  r+8\right)  r}.$ Indeed,%
\begin{align*}
\log k &  >\log n-\log r^{r^{4}p/\left(  p-1\right)  }>\frac{1}{2}\log
n+2r^{\left(  r+8\right)  r}\frac{p}{p-1}-r^{4}\frac{p}{p-1}\log r\\
&  >\frac{1}{2}\log n+\left(  2r^{\left(  r+8\right)  r-2}-r^{4}\right)
\frac{p}{p-1}\log r>\frac{1}{2}\log n>\frac{1}{c}>2r^{\left(  r+8\right)  r}.
\end{align*}
Now, setting $a:=2c,$ Corollary \ref{cor1} implies that $K_{r}^{+}\left(
\left\lfloor 2c\log k\right\rfloor \right)  \subset H.$ Since
\[
2c\log k>c\log n,
\]
Theorem \ref{sE} is proved if $1<p\leq2.$ Now, assume that $p>2.$ Then, in
view of (\ref{inc}) and (\ref{lv}), we have%
\[
\lambda^{\left(  2\right)  }\left(  G\right)  \geq\lambda^{\left(  p\right)
}\left(  G\right)  n^{2/p-1}\geq\lambda^{\left(  p\right)  }\left(
T_{r}\left(  n\right)  \right)  >\left(  1-\frac{1}{r}\right)  n-\frac{r}{4n},
\]
and applying Theorem \ref{t4} with $p=2,$ we get a subgraph $H\subset G$ of
order
\[
k>A^{p/\left(  \gamma p-\gamma\right)  }n=\left(  1-1/r\right)  ^{2r^{4}%
}n>\frac{n}{r^{2r^{4}}},
\]
with $\lambda^{\left(  p\right)  }\left(  H\right)  >\left(  1-1/r\right)  k$
and $\delta\left(  H\right)  >\left(  1-1/r-1/r^{4}\right)  k.$ Again we see
that $\log k>2r^{\left(  r+8\right)  r},$ due to
\begin{align*}
\log k &  >\log n-\log r^{2r^{4}}>\frac{1}{2}\log n+2r^{\left(  r+8\right)
r}\frac{p}{p-1}-2r^{4}\log r\\
&  >\frac{1}{2}\log n+\left(  2r^{\left(  r+8\right)  r-2}-2r^{4}\right)  \log
r>\frac{1}{2}\log n>\frac{1}{c}>2r^{\left(  r+8\right)  r}.
\end{align*}
Hence, setting $a:=2c,$ Corollary \ref{cor1} implies that $K_{r}^{+}\left(
\left\lfloor 2c\log k\right\rfloor \right)  \subset H.$ Since
\[
2c\log k>c\log n,
\]
Theorem \ref{sE} is proved for $p>2,$ as well.
\end{proof}

\medskip

\textbf{Acknowledgements\medskip}

The research of the first author was partially supported by the National
Natural Science Foundation of China (Nos. 11171207, 91130032). Part of this
work was done while the second author was visiting Shanghai University and
Hong Kong Polytechnic University in the Fall of 2013. He is grateful for the
outstanding hospitality of these universities. The authors are grateful to
Prof. Xiying Yuan for helpful discussions.


\bigskip

\begin{thebibliography}{99}                                                                                               %


\bibitem {Bol78}B. Bollob\'{a}s, \emph{Extremal Graph Theory,} Academic Press
Inc., London-New York, 1978, xx+488 pp.

\bibitem {Bol98}B. Bollob\'{a}s, \emph{Modern Graph Theory,} Graduate Texts in
Mathematics, \textbf{184,} Springer-Verlag, New York (1998), xiv+394 pp.

\bibitem {BoNi04}B. Bollob\'{a}s and V. Nikiforov, Joints in graphs,
\emph{Discrete Math. }\textbf{308} (2008), 9-19.

\bibitem {Erd63}P. Erd\H{o}s, On the structure of linear graphs, \emph{Israel
J. Math.} \textbf{1} (1963), 156--160.

\bibitem {Erd69}P. Erd\H{o}s, On the number of complete subgraphs and circuits
contained in graphs, \emph{\v{C}asopis P\v{e}st. Mat.} \textbf{94} (1969), 290--296.

\bibitem {ErSi73}P. Erd\H{o}s, M. Simonovits, On a valence problem in extremal
graph theory, $\emph{Dis}$\emph{$A$}$\emph{rete}$ $\emph{Math.}$ \textbf{5}
(1973), p. 323-334.

\bibitem {HLP88}G.H. Hardy, J.E. Littlewood, and G. P\'{o}lya,
\emph{Inequalities, 2nd edition}, Cambridge University Press, 1988, vi+324 pp.

\bibitem {HJZ11}B. He, Y.L. Jin and X.D. Zhang, Sharp bounds for the signless
Laplacian spectral radius in terms of clique number, \emph{\ Linear Algebra
Appl. }\textbf{438} (2013), 3851--3861.

\bibitem {KNY14}L. Kang, V. Nikiforov, and X.Yuan, The $p$-spectral radius of
$k$-partite and $k$-chromatic uniform hypergraphs. Preprint available at
\emph{arXiv:1402.0442.}

\bibitem {KLM13}P. Keevash, J. Lenz, and D. Mubayi, Spectral extremal problems
for hypergraphs. Preprint available at\emph{ arXiv:1304.0050.}

\bibitem {MoSt65}T. Motzkin and E. Straus, Maxima for graphs and a new proof
of a theorem of Tur\'{a}n. \emph{Canad. J. Math.,} \textbf{17}, (1965), 533-540.

\bibitem {Nik02}V. Nikiforov, Some inequalities for the largest eigenvalue of
a graph, \emph{Combin. Probab. Comput.} \textbf{11} (2002), 179-189.

\bibitem {Nik07}V. Nikiforov, Bounds on graph eigenvalues II, \emph{Linear
Algebra Appl.} \textbf{427} (2007), 183-189.

\bibitem {Nik08}V. Nikiforov, A spectral condition for odd cycles,
\emph{Linear Algebra Appl.} \textbf{428} (2008), 1492-1498.

\bibitem {Nik09b}V. Nikiforov, Spectral saturation: inverting the spectral
Tur\'{a}n theorem, \emph{Electronic J. Combin.} \textbf{15} (2009), R33.

\bibitem {Nik10}V. Nikiforov, Tur\'{a}n's theorem inverted, \emph{Discrete
Math.} \textbf{310}, (2010), 125-131.

\bibitem {Nik11}V. Nikiforov, Some new results in extremal graph theory, in
\emph{Surveys in Combinatorics,} Cambridge University Press, 2011, pp. 141--181.

\bibitem {Nik14}V. Nikiforov, Some extremal problems for hereditary properties
of graphs, \emph{Electron. J. Combin. }\textbf{21}, (2014), P1.17.

\bibitem {NikA}V. Nikiforov, An analytic theory of extremal hypergraph
problems. Preprint available at\emph{ arXiv:1305.1073v2}.

\bibitem {NikB}V. Nikiforov, Analytic methods for uniform hypergraphs.
Preprint available at\emph{ arXiv:1308.1654v3}.

\bibitem {Sim68}M. Simonovits, A method for solving extremal problems in graph
theory, stability problems, in: \emph{Theory of Graphs (Proc. Colloq., Tihany,
1966),} pp. 279--319, Academic Press, New York, 1968.

\bibitem {Tur41}P. Tur\'{a}n, On an extremal problem in graph theory (in
Hungarian), $\emph{Mat.}$ \emph{\'{e}s Fiz. Lapok} \textbf{48 }(1941) 436-452.

\bibitem {Wil86}H. Wilf, Spectral bounds for the clique and independence
numbers of graphs, \emph{J. Combin. Theory Ser. B} \textbf{40} (1986), 113-117.

\bibitem {Zyk49}A. A. Zykov, On some properties of linear complexes (in
Russian), \emph{Mat. Sbornik N.S.} \textbf{24}(66), (1949), 163--188.
\end{thebibliography}
\end{document}